\documentclass[11pt]{amsart}
\usepackage{amsfonts}
\usepackage{amssymb}
\usepackage{eurosym}
\usepackage{amsfonts}
\usepackage{amsmath}
\usepackage{enumerate}
\usepackage{color}
\usepackage{cite}
\usepackage{subfigure}
\usepackage{float}
\usepackage[colorlinks]{hyperref}
\usepackage{paralist}
\usepackage{graphicx}
\newtheorem{theorem}{Theorem}[section]

\newtheorem{lemma}[theorem]{Lemma}

\theoremstyle{definition}

\newtheorem*{algorithm}{Algorithm}

\newtheorem*{problem}{Problem}
\newtheorem{remark}{Remark}[section]
\numberwithin{equation}{section}
\newtheorem*{examA}{Example~A}
\newtheorem*{examB}{Example~B}
\newtheorem*{numerical example}{Numerical Results}
\numberwithin{equation}{section} \textwidth=16cm \textheight=21cm
\begin{document}
\title[Control of semilinear differential equations with moving singularities%
]{Control of semilinear differential equations with moving singularities}
\author[R. Precup]{Radu Precup}
\address{R. Precup, Faculty of Mathematics and Computer Science and
Institute of Advanced Studies in Science and Technology, Babe\c{s}-Bolyai
University, 400084 Cluj-Napoca, Romania \& Tiberiu Popoviciu Institute of
Numerical Analysis, Romanian Academy, P.O. Box 68-1, 400110 Cluj-Napoca,
Romania}
\email{r.precup@ictp.acad.ro}
\author[A. Stan]{Andrei Stan}
\address{A. Stan, Department of Mathematics, Babe\c{s}-Bolyai University,
400084 Cluj-Napoca, Romania \& Tiberiu Popoviciu Institute of Numerical
Analysis, Romanian Academy, P.O. Box 68-1, 400110 Cluj-Napoca, Romania}
\email{andrei.stan@ubbcluj.ro}
\author[W-S. Du]{Wei-Shih Du}
\address{W-S. Du, Department of Mathematics, National Kaohsiung Normal University,
Kaohsiung 82444, Taiwan}
\email{wsdu@mail.nknu.edu.tw}
\maketitle

\begin{abstract}
In this paper, we present a control problem related to a semilinear differential equation with a moving singularity, i.e., the singular point depends on a parameter. The particularity of the controllability condition resides in the fact that it depends on the singular point which in turn depends on the control variable. We provide sufficient conditions to ensure that the functional determining the control is continuous over the entire
domain of the parameter. Lower and upper solutions technique combined with a bisection algorithm is used to prove the controllability of the equation and to approximate the control. An example is given together with some numerical simulations.  The results naturally extend to fractional differential equations.
\end{abstract}



\section{Introduction and preliminaries}

Differential equations are crucial in solving practical problems in many scientific fields, such as physics, chemistry, biology, economics, and engineering, etc., modeling  many real-world processes. However,
the complexity of these phenomena often introduces various parameters that can significantly influence the outcome. A particularly intriguing problem in this context is identifying the parameters that ensure a specific quantity (e.g., density, energy) related to the solution of the differential
equation reaches a desired value. This challenge naturally leads to a control problem.

Our study has two strong motivations. 
\begin{description}
    \item[(a)] The first motivation concerns differential equations with moving singularities, which frequently appear in nonlinear models from applied sciences, such as physics and mathematical biology \cite{tores}. 
    
    \item[(b)] The second one relates to the control of such models, aiming to reach a desired state of the process. For example, if the state variable represents a density, one might be interested in controlling its cumulative value or average. This corresponds precisely to our control problem in Section 2.
\end{description}

Mathematical models expressed through equations with singularities include the Briot-Bouquet equation, which has applications in complex analysis, specifically in the theory of univalent functions; equations arising in Michaelis-Menten kinetics, modeling oxygen diffusion in cells; the Thomas-Fermi equation in atomic physics; and the Emden-Fowler equation,  in the study of phenomena in non-Newtonian fluid mechanics \cite{tores,oregan,doi}.

 Inspired by the investigation in \cite{moving1,moving2,movingsolution}, this paper will explore the following problem
\begin{equation}
\begin{cases}
u^{\prime }(t)=f(t,u(t),\lambda ),\quad t\in \left[ 0,\,\theta (\lambda
)\right)  \\
u(0)=u_{0}(\lambda ).
\end{cases}
\label{pb principala}
\end{equation}%
 Here,
 $\ f:[0,\infty)\times\mathbb{R}\times(0,\infty)\rightarrow\Bar{\mathbb{R}}(:=\mathbb{R\cup\{\pm\infty\}})$ is a function that possesses a singularity in the first variable, influenced by the third one, that is, for each $\lambda >0$, there exists $\theta
(\lambda )>0$ such that
\begin{equation*}
\lim_{\substack{ t\rightarrow \,\theta (\lambda ) \\ t<\,\theta (\lambda )}}%
f(t,x,\lambda )=\pm \infty ~\quad \text{for almost all }x\in \mathbb{R}.
\end{equation*}
Throughout this paper, we use $u_{\lambda }$ to denote the unique
solution to problem (\ref{pb principala}) for a given $\lambda >0$.
    Since the singularity point $\theta(\lambda)$ varies  with $\lambda$, the differential equation \eqref{pb principala} is said to be with moving singularity.

For each $\lambda >0$, we consider the functional $\psi _{\lambda }\colon
C[0,\theta (\lambda ))\rightarrow \Bar{\mathbb{R}}$,
\begin{equation}
\psi _{\lambda }(u)=\int_{0}^{\theta (\lambda )}u(s)\,ds,  \label{def psi}
\end{equation}%
where the integration over the noncompact interval $\left[ 0,\,\theta
(\lambda )\right) $ is understood in the usual sense (see, e.g., \cite%
{analizaII}),
\begin{equation*}
\int_{0}^{\theta (\lambda )}u(s)\,ds=\lim_{\substack{ t\rightarrow \theta
(\lambda )  \\ t<\theta (\lambda )}}\int_{0}^{t}u(s)\,ds.
\end{equation*}

Our goal in this paper is the following \textit{control problem}:

\begin{problem}[control problem]
Find $\lambda ^{\ast }>0$ such that
\begin{equation}
\psi _{\lambda ^{\ast }}(u_{\lambda ^{\ast }})=p,  \label{cc}
\end{equation}%
where $p\in \mathbb{R}$ is a given value.
\end{problem}
The particularity of the controllability condition resides in the fact that
it depends on the singular point which in turn depends on the control
variable.
To establish sufficient conditions for the existence of a solution to this
control problem, we first guarantee that the mapping $\varphi \colon \left(
0,\infty \right) \rightarrow \mathbb{R}$,
\begin{equation}
\varphi (\lambda )=\psi _{\lambda }(u_{\lambda }),  \label{phi}
\end{equation}%
is well-defined and continuous on $(0,\infty )$. Then, we are able to use a
lower and upper solution argument to guarantee the existence of a $\lambda
^{\ast }$ with the desired property (\ref{cc}). Moreover, by 
bisection algorithm, we have a method of approximation of the value $\lambda^\ast.$

Control problems on a fixed interval that require to find a parameter in
order to achieve a specific controllability condition are well-documented in
the literature (see, e.g., \cite{pc, pc2, coron}). The novelty of this paper
lies in determining such a parameter where each solution is defined on a
different interval. This approach requires a more refined analysis and leads
to more complex problems. 

The main assumptions we use in our analysis are

\begin{itemize}
\item[\textbf{(h1)}] For each $\lambda >0$ and $0<\varepsilon <\theta
(\lambda )$, there exists a constant $L=L(\lambda ,\varepsilon )$ such that
for all $x,\overline{x}\in \mathbb{R}$ and for all $t\in \left[ 0,\,\theta
(\lambda )-\varepsilon \right] $, we have
\begin{equation*}
|f(t,x,\lambda )-f(t,\overline{x},\lambda )|\leq L(\lambda ,\varepsilon
)\,|x-\overline{x}|.
\end{equation*}

\item[\textbf{(h2)}] The mappings
\begin{equation*}
\lambda \mapsto \theta (\lambda ),\quad \lambda \mapsto u_{0}(\lambda
),\quad \lambda \mapsto L(\lambda ,\varepsilon )\,\,(\text{for each }%
\varepsilon >0),
\end{equation*}%
are continuous. Additionally, the map $(t,x,\lambda )\mapsto f(t,x,\lambda )$
is continuous for $t\in \lbrack 0,\,\theta (\lambda ))$, $x\in \mathbb{R}$,
and $\lambda \in (0,\infty )$.
\end{itemize}

In the next lemma, we show that assumptions (h1), (h2) are sufficient to
guarantee the existence of a unique solution of problem (\ref{pb principala}%
).

\begin{lemma}
Under assumptions \emph{(h1)} and \emph{(h2)}, for each $\lambda >0$, there
exists a unique solution $u_{\lambda }\in C[0,\,\theta (\lambda ))$ of
problem (\emph{\ref{pb principala}}$)$. Moreover, this solution satisfies
the integral equation
\begin{equation}
u_{\lambda }(t)=u_{0}(\lambda )+\int_{0}^{t}f(s,u_{\lambda }(s),\lambda
)\,ds,  \label{forma integrala}
\end{equation}%
for all $t\in \left[ 0,\,\theta (\lambda )\right) .$
\end{lemma}

\begin{proof}
Let $\lambda >0$. For each $0<\varepsilon <\theta (\lambda )$, we consider
the initial value problem on the cut-off domain,
\begin{equation}
\begin{cases}
u^{\prime }(t)=f(t,u(t),\lambda ), & t\in \left[ 0,\,\theta (\lambda
)-\varepsilon \right]  \\
u(0)=u_{0}(\lambda ). &
\end{cases}
\label{pb_epsilon}
\end{equation}%
From assumption (h1), the mapping $x\mapsto f(t,x,\lambda )$ is Lipschitz
continuous with the Lipschitz constant $L_{\lambda ,\varepsilon }$.
Therefore, problem (\ref{pb_epsilon}) has a unique solution (see, e.g., \cite%
{pp,Coddington}). The conclusion follows immediately by letting $\varepsilon
\rightarrow 0$ and using the uniqueness of the solution on each interval $%
[0,\theta (\lambda )-\varepsilon ]$. Relation (\ref{forma
integrala}) can be easily deduced by taking the integral in (\ref{pb principala})
from $0$ to $t.$
\end{proof}

\begin{examA}
The model for $f(t,x,\lambda )$ is given by
\begin{equation*}
f(t,x,\lambda )=\frac{x}{\lambda -t}.
\end{equation*}%
Clearly, there is a singularity at $\theta (\lambda )=\lambda $. However,
when the first variable is restricted to a compact interval $[0,\lambda
-\varepsilon ]$ with $0<\varepsilon <\theta (\lambda )$, the function $f$ is Lipschitz
continuous with $L(\lambda ,\varepsilon )=\frac{1}{\varepsilon }$ the
Lipschitz constant. If we set $u_{0}(\lambda )=\frac{1}{\lambda }>0$, the
unique solution of the problem \eqref{pb principala} is
\begin{equation*}
u_{\lambda }(t)=\frac{1}{\lambda -t}\,,\ \ \text{for all }t\in \lbrack
0,\lambda ).
\end{equation*}
\end{examA}

\begin{remark}
\label{remarca_phi} Without imposing further assumptions beyond (h1) and
(h2), we cannot generally expect $\varphi $ to be well-defined and
continuous on $(0,\infty )$. For instance, consider the case where
\begin{equation*}
f(t,x,\lambda )=\frac{1}{\lambda (\lambda -t)^{\frac{1}{\lambda }+1}}\quad
\text{and}\quad u_{0}(\lambda )=\frac{1}{\lambda ^{1/\lambda }}.
\end{equation*}%
Clearly, both assumptions (h1) and (h2) are satisfied. Straightforward
calculations yield
\begin{equation*}
u_{\lambda }(t)=(\lambda -t)^{-\frac{1}{\lambda }},
\end{equation*}%
for each $\lambda >1$. Consequently,
\begin{equation*}
\varphi (\lambda )=\frac{1}{\lambda -1}\lambda ^{\frac{2\lambda -1}{\lambda }%
}.
\end{equation*}%
This implies that $\varphi (\lambda )\in \mathbb{R}$ for $\lambda >1$, but
\begin{equation*}
\lim_{\substack{ \lambda \rightarrow 1  \\ \lambda >1}}\varphi (\lambda
)=\infty .
\end{equation*}
\end{remark}

The paper is structured as follows: Section 2 presents the original results on the controllability of equations with moving singularities. Before establishing the main result (Theorem \ref{th 26}), we derive several auxiliary results concerning the continuous dependence of solutions on the control parameter and the continuity of the control functional with respect to the control variable. In Section 3, we provide a theoretical algorithm for obtaining the solution of the control problem using the method of lower and upper solutions. We also provide an example that illustrates the applicability of the obtained results. Finally, in Section 4 we suggest a possible extension of our approach to fractional differential equations with moving singularities.

In the following, we present some well-known results from the literature
that will be used throughout this paper. The first result is the Arzel\`{a}%
-Ascoli theorem (see, e.g., \cite{p,ciarlet}).

\begin{theorem}[Arzel\`{a}-Ascoli theorem]
A subset $F\subset \left( C[a,b],|\cdot |_{\infty }\right) $, where
$|\cdot |_{\infty }$ is the supremum norm, is relatively compact if
and only if it is uniformly bounded and equicontinuous, that is,
there exists $M>0$ with
\begin{equation*}
|u|_{\infty }\leq M,\quad \text{for all }u\in F,
\end{equation*}%
and for every $\varepsilon >0$, there exists $\delta (\varepsilon )>0$ such
that
\begin{equation*}
\left\vert u(x)-u(y)\right\vert \leq \varepsilon ,\quad \text{for }x,y\in
\lbrack a,b]\text{ with }\left\vert x-y\right\vert \leq \delta (\varepsilon )%
\text{ and all }u\in F.
\end{equation*}
\end{theorem}

The next lemma provides an alternative condition to ensure the convergence
of a sequence based on the behavior of its subsequences (see, e.g., \cite[%
Lemma~1.1]{herve}).

\begin{theorem}
\label{teorema_sir_covnergent} Let $X$ be a topological space, and let $%
(x_n) $ be a sequence in $X$ with the following property: there exists $x
\in X$ such that from any subsequence of $(x_n)$, a further subsequence can
be extracted that converges to $x$. Then the entire sequence $(x_n)$
converges to $x$.
\end{theorem}

Finally, we have the well known Gr\"{o}nwall's lemma (see, \cite{gronwall,ineq}).

\begin{theorem}[Gr\"{o}nwall inequality]
Let $u,v\in C[a,b]$. If there exists a constant $c>0$ such that
\begin{equation*}
\left\vert u(t)\right\vert \leq c+\int_{a}^{t}|u(s)|\,|v(s)|\,ds~\quad \text{%
for }t\in \lbrack a,b],
\end{equation*}%
then
\begin{equation*}
\left\vert u(t)\right\vert \leq c\exp \left( \int_{a}^{t}|v(s)|\,ds\right)
~\quad \text{for }t\in \lbrack a,b].
\end{equation*}
\end{theorem}

\section{Auxiliary lemmas and new controllability result}
\smallskip
\subsection{Proprieties of the solutions $u_\protect\lambda$}

In this section, we present some properties of the solutions to problem %
\eqref{pb principala}, which will be used later in our analysis.

\begin{lemma}
\label{th1} Assume that assumptions \emph{(h1)} and \emph{(h2)} hold. Then,
given any sequence of positive real numbers $\left( \lambda _{k}\right) $
converging to some $\lambda ^{\ast }>0$, the sequence of solutions $\left(
u_{\lambda _{k}}\right) $ converges uniformly to $u_{\lambda ^{\ast }}$ on
any interval $[0,\,\theta (\lambda ^{\ast })-\varepsilon ]$, with $%
0<\varepsilon <\min \left\{ \inf_{k\in \mathbb{N}}\theta \left( \lambda
_{k}\right) ,\theta (\lambda ^{\ast })\right\} $.
\end{lemma}

\begin{proof}
Let $\left( \lambda _{k}\right) $ be a sequence of positive real numbers
with $\lambda _{k}\rightarrow \lambda ^{\ast }>0$ as $k\rightarrow \infty $,
and let
\begin{equation}
0<\varepsilon <\min \left\{ \inf_{k\in \mathbb{N}}\theta \left( \lambda
_{k}\right) ,\theta (\lambda ^{\ast })\right\} .  \label{alegere_epsilon}
\end{equation}%
Since the mapping $\lambda \mapsto \theta (\lambda )$ is continuous, there
exists $k_{0}\in \mathbb{N}$ such that
\begin{equation*}
\theta (\lambda ^{\ast })-\varepsilon \leq \theta \left( \lambda _{k}\right)
-\frac{\varepsilon }{2},
\end{equation*}%
for all $k\geq k_{0}$. Clearly, the choise of $\varepsilon $ given in $%
\left( \ref{alegere_epsilon}\right) $ guarantees that both $\theta (\lambda
^{\ast })-\varepsilon \ $and  $\theta \left( \lambda _{k}\right) -\frac{%
\varepsilon }{2}$ are strictly positive, for all $k\geq k_{0}.$

Using condition (h1), for each $k\geq k_{0}$, there exists $L\left( \lambda
_{k},\frac{\varepsilon }{2}\right) >0$ such that
\begin{equation*}
\left\vert f\left( t,x,\lambda _{k}\right) -f\left( t,\overline{x},\lambda
_{k}\right) \right\vert \leq L\left( \lambda _{k},\frac{\varepsilon }{2}%
\right) \left\vert x-\overline{x}\right\vert
\end{equation*}%
for all $t\in \lbrack 0,~\theta (\lambda _{k})-\varepsilon ]$ and all $x,%
\overline{x}\in \mathbb{R}$. Moreover, from the continuity of the mapping $%
\lambda \mapsto L\left( \lambda ,\frac{\varepsilon }{2}\right) $, there
exists a positive real number $\Bar{L}>0$ such that
\begin{equation*}
L\left( \lambda _{k},\frac{\varepsilon }{2}\right) \leq \Bar{L}\quad \text{%
for all }k\geq k_{0}.
\end{equation*}%
Furthermore, using the Arzel\`{a}-Ascoli theorem, we show that the set
\begin{equation*}
A:=\left\{ u_{\lambda _{k}}:[0,\theta \left( \lambda ^{\ast }\right)
-\varepsilon ]\rightarrow \mathbb{R},\ k\geq k_{0}\right\}
\end{equation*}%
is relatively compact in $C[0,\theta \left( \lambda ^{\ast }\right)
-\varepsilon ]$. First, we establish the uniform boundedness of the set $A$.
Let $k\geq k_{0}$ and let $t\in \lbrack 0,\,\theta \left( \lambda ^{\ast
}\right) -\varepsilon ]$. Using the integral form (\ref{forma integrala}) of
$u_{\lambda _{k}}$, we obtain

%\begin{adjustwidth}{-1.5cm}{0cm}
\begin{eqnarray}
\left\vert u_{\lambda _{k}}(t)\right\vert  &\leq &\left\vert u_{0}(\lambda
_{k})\right\vert +\int_{0}^{t}\left\vert f(s,u_{\lambda _{k}}(s),\lambda
_{k})-f(s,0,\lambda _{k})\right\vert ds+\int_{0}^{t}\left\vert f(s,0,\lambda
_{k})\right\vert ds  \label{marg_superioara_1} \\
&\leq &\left\vert u_{0}(\lambda _{k})\right\vert +\int_{0}^{t}e^{s\rho
}e^{-s\rho }\Bar{L}\,\left\vert u_{\lambda _{k}}(s)\right\vert
ds+\int_{0}^{t}\left\vert f(s,0,\lambda _{k})\right\vert ds  \notag \\
&\leq &\left\vert u_{\lambda _{k}}\right\vert _{\rho }\frac{\Bar{L}}{\rho }%
e^{t\rho }+M,  \notag
\end{eqnarray}%
%\end{adjustwidth}

\noindent where
\begin{equation*}
\left\vert u\right\vert _{\rho }=\sup_{t\in \lbrack 0,\,\,\theta \left(
\lambda ^{\ast }\right) -\varepsilon ]}e^{-t\rho }\left\vert u(t)\right\vert
,
\end{equation*}%
is the Bielecki norm on $C[0,\theta \left( \lambda ^{\ast }\right)
-\varepsilon ]$, and
\begin{equation*}
M=\sup_{k\geq k_{0}}\left( \left\vert u_{0}(\lambda _{k})\right\vert
+\int_{0}^{\theta \left( \lambda ^{\ast }\right) -\varepsilon }\left\vert
f\left( s,0,\lambda _{k}\right) \right\vert ds\right) .
\end{equation*}%
Note that $M<\infty $. Indeed, the continuity of the mapping $(t,\lambda
)\mapsto f(t,0,\lambda )$ guarantees that it is uniformly bounded on the
compact set
\begin{equation*}
(t,\lambda )\in \left[ 0,\theta \left( \lambda ^{\ast }\right) -\varepsilon %
\right] \times \left[ \inf_{k\geq k_{0}}\lambda _{k},\sup_{k\geq
k_{0}}\lambda _{k}\right] ,
\end{equation*}%
since $\inf_{k\geq k_{0}}\lambda _{k}>0\text{ and }\sup_{k\geq k_{0}}\lambda
_{k}<\infty.$ Dividing (\ref{marg_superioara_1}) by $e^{t\rho }$ and taking
the supremum over $[0,\,\theta \left( \lambda ^{\ast }\right) -\varepsilon ]$%
, we obtain
\begin{equation}
\left\vert u_{\lambda _{k}}\right\vert _{\rho }\leq \frac{\Bar{L}}{\rho }%
\left\vert u_{\lambda _{k}}\right\vert _{\rho }+M.  \label{eq3}
\end{equation}%
Thus, if we choose $\rho >\Bar{L}$, relation \eqref{eq3} ensures the uniform
boundedness of the sequence $\left( u_{\lambda _{k}}\right) _{k\geq k_{0}}$. Let $D$ be an upper bound for this sequence, that is,
\begin{equation*}
D=\sup_{\substack{ t\in \lbrack 0,\,\theta (\lambda ^{\ast })-\varepsilon ]
\\ k\geq k_{0}}}\left\vert u_{\lambda _{k}}(t)\right\vert .
\end{equation*}

\noindent For any $t,\overline{t}\in \lbrack 0,\,\theta (\lambda ^{\ast })-\varepsilon
]$ and $k\geq k_{0}$, one has
\begin{equation*}
\left\vert u_{\lambda _{k}}(t)-u_{\lambda _{k}}(\overline{t})\right\vert
\leq \left\vert \int_{\overline{t}}^{t}\left\vert f(s,u_{\lambda
_{k}}(s),\lambda _{k})\right\vert ds\right\vert \leq \Bar{M}\left\vert t-%
\overline{t}\right\vert ,
\end{equation*}%
where
\begin{equation*}
\Bar{M}=\sup_{\substack{ t\in \lbrack 0,\,\theta (\lambda ^{\ast
})-\varepsilon ]  \\ \left\vert x\right\vert \leq D,\,k\geq k_{0}}}%
\left\vert f(t,x,\lambda _{k})\right\vert .
\end{equation*}%
Consequently, the set $A$ is equicontinuous. By the Arzel\`{a}-Ascoli
theorem, there exists a subsequence $\left( u_{\lambda _{k_{n}}}\right) $ of
$\left( u_{\lambda _{k}}\right) $ that converges uniformly to  $%
u^{\ast }$ on  $[0,\,\theta \left( \lambda ^{\ast }\right) -\varepsilon ]$.

We now show that
\begin{equation*}
u^{\ast }(t)=u_{\lambda ^{\ast }}(t)~\quad \text{for all }t\in \lbrack
0,\,\theta \left( \lambda ^{\ast }\right) -\varepsilon ].
\end{equation*}%
Let $t\in \lbrack 0,\,\theta \left( \lambda ^{\ast }\right) -\varepsilon ]$.
Then
\begin{align}
\left\vert u^{\ast }(t)-u_{\lambda ^{\ast }}(t)\right\vert & \leq \left\vert
u_{\lambda _{k_{n}}}(t)-u_{\lambda ^{\ast }}(t)\right\vert +\left\vert
u_{\lambda _{k_{n}}}(t)-u^{\ast }(t)\right\vert  \label{eq4} \\
& \leq \left\vert u_{0}\left( \lambda _{k_{n}}\right) -u_{0}\left( \lambda
^{\ast }\right) \right\vert +\int_{0}^{t}\left\vert f\left( s,u_{\lambda
_{k_{n}}}(s),\lambda _{k_{n}}\right) -f\left( s,u_{\lambda ^{\ast
}}(s),\lambda ^{\ast }\right) \right\vert \,ds  \notag \\
& \quad +\left\vert u_{\lambda _{k_{n}}}(t)-u^{\ast }(t)\right\vert  \notag
\\
& \leq \int_{0}^{t}\left\vert f\left( s,u_{\lambda _{k_{n}}}(s),\lambda
_{k_{n}}\right) -f\left( s,u_{\lambda ^{\ast }}(s),\lambda _{k_{n}}\right)
\right\vert \,ds  \notag \\
& \quad +\int_{0}^{t}\left\vert f\left( s,u_{\lambda ^{\ast }}(s),\lambda
_{k_{n}}\right) -f\left( s,u_{\lambda ^{\ast }}(s),\lambda ^{\ast }\right)
\right\vert \,ds  \notag \\
& \quad +\left\vert u_{\lambda _{k_{n}}}(t)-u^{\ast }(t)\right\vert
+\left\vert u_{0}\left( \lambda _{k_{n}}\right) -u_{0}\left( \lambda ^{\ast
}\right) \right\vert  \notag \\
& \leq \frac{e^{\rho t}}{\rho }\Bar{L}\left\vert u_{\lambda
_{k_{n}}}-u_{\lambda ^{\ast }}\right\vert _{\rho }+\int_{0}^{t}\left\vert
f\left( s,u_{\lambda ^{\ast }}(s),\lambda _{k_{n}}\right) -f\left(
s,u_{\lambda ^{\ast }}(s),\lambda ^{\ast }\right) \right\vert \,ds  \notag \\
& \quad +\left\vert u_{\lambda _{k_{n}}}(t)-u^{\ast }(t)\right\vert
+\left\vert u_{0}\left( \lambda _{k_{n}}\right) -u_{0}\left( \lambda ^{\ast
}\right) \right\vert .  \notag
\end{align}

\noindent Multiplying (\ref{eq4}) by $e^{-\rho t}$ and taking the supremum over $%
[0,\theta \left( \lambda ^{\ast }\right) -\varepsilon ],$ we find

\begin{eqnarray}
\left\vert u^{\ast }-u_{\lambda ^{\ast }}\right\vert _{\rho } &\leq &\frac{%
\Bar{L}}{\rho }\left\vert u_{\lambda _{k_{n}}}-u_{\lambda ^{\ast
}}\right\vert _{\rho }+\int_{0}^{t}\left\vert f\left( s,u_{\lambda ^{\ast
}}(s),\lambda _{k_n}\right) -f\left( s,u_{\lambda ^{\ast }}(s),\lambda ^{\ast
}\right) \right\vert ds  \label{eq55} \\
&&+\left\vert u_{\lambda _{k_{n}}}-u^{\ast }\right\vert _{\infty
}+\left\vert u_{0}\left( \lambda _{k_{n}}\right) -u_{0}\left( \lambda ^{\ast
}\right) \right\vert .  \notag
\end{eqnarray}%

\noindent Since the last three terms from the right hand side of \eqref{eq55} tends to
zero and
$$\left\vert u_{\lambda _{k_{n}}}-u_{\lambda ^{\ast }}\right\vert
_{\rho }\rightarrow \left\vert u^{\ast }-u_{\lambda ^{\ast }}\right\vert
_{\rho },$$
choosing $\rho >\Bar{L}$, we find that $\left\vert u^{\ast
}-u_{\lambda ^{\ast }}\right\vert _{\rho }=0,$ which means that
\[
u_{\lambda^{\ast}}(t)=u^{\ast}(t)\quad\text{for all }t\in\lbrack0,\,\theta
(\lambda^{\ast})-\varepsilon].
\]

The conclusion of our theorem now follows directly from Theorem \ref%
{teorema_sir_covnergent}. Indeed, for any subsequence $\left( u_{\lambda
_{k}}\right) $, using the same procedure described above, we can extract a
further subsequence that converges to $u_{\lambda ^{\ast }}$ on the interval
$[0,\theta (\lambda ^{\ast })-\varepsilon ]$. Therefore, Theorem \ref%
{teorema_sir_covnergent} applies and guarantees that on $[0,\theta (\lambda
^{\ast })-\varepsilon ]$, the entire sequence $\left( u_{\lambda
_{k}}\right) $ converges to $u_{\lambda ^{\ast }}$.
\end{proof}

For any $\varepsilon > 0$, we denote
\begin{equation*}
B_{\varepsilon } := \theta^{-1} \left( [2\varepsilon, \infty) \right) = \left\{ \lambda > 0 \,:\, \theta(\lambda) \geq 2\varepsilon \right\}.
\end{equation*}
In the following two results, for some $\varepsilon > 0$ and $\lambda \in B_\varepsilon$, we establish certain properties of the solutions $u_\lambda$ over the interval $[0, \theta(\lambda) - \varepsilon]$. Since $\lambda$ belongs to $B_\varepsilon$, it follows that $\theta(\lambda) - \varepsilon \geq \varepsilon > 0$.
% By the next lemma, we establish upper bounds for the solutions $u_{\lambda }$
% on compact intervals.

\begin{lemma}
\label{lema_aux1} Assume that conditions \emph{(h1)} and \emph{(h2)} hold.
Then, for each $\varepsilon >0$, there is a continuous function
\begin{equation*}
\tau _{\varepsilon }\colon B_{\varepsilon }\rightarrow \left( 0,\infty
\right) ,
\end{equation*}%
such that
\begin{equation*}
\left\vert u_{\lambda }\left( t\right) \right\vert \leq \tau _{\varepsilon
}\left( \lambda \right)\quad \text{for all $t\in \left[ 0,\,\theta (\lambda
)-\varepsilon \right] $},
\end{equation*}%
and every $\lambda \in B_{\varepsilon }.$
\end{lemma}

\begin{proof}
Let $\varepsilon >0$ be fixed with $B_{\varepsilon }$ nonempty, $\lambda \in
B_{\varepsilon }$ and $t\in \lbrack 0,\,\theta (\lambda )-\varepsilon ]$.
Using the integral form (\ref{forma integrala}) of $u_{\lambda }$,
straightforward computations yield,
\begin{equation*}
\left\vert u_{\lambda }(t)\right\vert \leq M_{1}+\int_{0}^{t}\left\vert
f(s,u_{\lambda }(s),\lambda )-f(s,0,\lambda )\right\vert ,
\end{equation*}%
where
\begin{equation*}
M_{1}=\left\vert u_{0}(\lambda )\right\vert +\int_{0}^{\theta (\lambda
)-\varepsilon }\left\vert f(s,0,\lambda )\right\vert \,ds.
\end{equation*}%
From (h2), there exists $L(\lambda ,\varepsilon )>0$ such that
\begin{equation*}
\left\vert f(s,x,\lambda )-f(s,0,\lambda )\right\vert \leq L(\lambda
,\varepsilon )\left\vert x\right\vert ,
\end{equation*}%
for all $s\in \lbrack 0,\theta (\lambda )-\varepsilon ]$ and all $x\in
\mathbb{R}.$ Thus,
\begin{equation}
\left\vert u_{\lambda }(t)\right\vert \leq M_{1}+L(\lambda ,\varepsilon
)\int_{0}^{t}\left\vert u_{\lambda }(s)\right\vert \,ds.
\label{ineq_u_lambda}
\end{equation}%
Using Gr\"{o}nwall's lemma in (\ref{ineq_u_lambda}) gives
\begin{equation*}
\left\vert u_{\lambda }(t)\right\vert \leq M_{1}e^{(\theta (\lambda
)-\varepsilon )L(\lambda ,\varepsilon )}:=\tau _{\varepsilon }(\lambda )\quad
\text{for all $t\in \lbrack 0,\theta (\lambda )-\varepsilon ]$,}
\end{equation*}%
while the continuity of $\lambda \mapsto \theta (\lambda ),\ L(\lambda
,\varepsilon )$ ensures that $\tau _{\varepsilon }$ is continuous.
\end{proof}

In the next result, we prove an important relation between
the solutions $u_\lambda$ and the integral.

\begin{theorem}
For every $\varepsilon >0$, the mapping $\ \xi \colon B_{\varepsilon
}\rightarrow \mathbb{R}$,
\begin{equation*}
\xi (\lambda ):=\int_{0}^{\theta (\lambda )-\varepsilon }u_{\lambda }(s)\,ds,
\end{equation*}%
is continuous.
\end{theorem}

\begin{proof}
Let $\varepsilon >0$ such that $B_{\varepsilon }$ is nonempty, and let $%
\left( \lambda _{k}\right) \subset B_{\varepsilon }$ be a sequence with $%
\lambda _{k}\rightarrow \lambda $ as $k\rightarrow \infty $. Clearly, $%
\lambda \in B_{\varepsilon }$ as the set $B_{\varepsilon }$ is closed.
Lemma \ref{th1} guarantees that $u_{\lambda _{k}}$ converges uniformly to $%
u_{\lambda }$ on the interval $[0,\,\theta (\lambda )-\varepsilon ]$.
Therefore, using Lemma \ref{lema_aux1}, we have
\begin{eqnarray}
\left\vert \xi \left( \lambda _{k}\right) -\xi (\lambda )\right\vert  &\leq
&\int_{0}^{\theta (\lambda )-\varepsilon }\left\vert u_{\lambda
_{k}}(s)-u_{\lambda }(s)\right\vert ds+\left\vert \int_{\theta (\lambda
)-\varepsilon }^{\theta \left( \lambda _{k}\right) -\varepsilon }\left\vert
u_{\lambda _{k}}(s)\right\vert ds\right\vert   \label{eq66} \\
&\leq &\int_{0}^{\theta (\lambda )-\varepsilon }\left\vert u_{\lambda
_{k}}(s)-u_{\lambda }(s)\right\vert ds+\tau _{\varepsilon }\left( \lambda
_{k}\right) \left\vert \theta \left( \lambda _{k}\right) -\theta (\lambda
)\right\vert .  \notag
\end{eqnarray}%
Since the continuity of $\tau _{\varepsilon }$ ensures that the sequence $%
\left( \tau _{\varepsilon }\left( \lambda _{k}\right) \right) $ is uniformly
bounded, we deduce that the right-hand side of inequality (\ref{eq66})
converges to zero, implying that $\xi \left( \lambda _{k}\right) $ converges
to $\xi (\lambda )$. Finally, since the above result holds for any sequence
within $B_{\varepsilon }$, we conclude our proof.
\end{proof}

\subsection{The continuity of $\protect\varphi$.}

In this subsection, we prove that the mapping $\varphi $ defined in (\ref{phi})
is well defined and continuous. As noted in Remark \ref{remarca_phi},
assumptions (h1) and (h2) alone are insufficient to guarantee that $\varphi $
is well-defined on $(0,\infty )$. Thus, we introduce the next growth
condition on the function $f$,  aligning it with the typical model described in Example B from below.
\begin{enumerate}
\item[\textbf{(h3)}] There exists a constant $a>1$ such that, for all $%
\lambda >0$, one has
\begin{equation*}
\left\vert f(t,x,\lambda )\right\vert \leq \frac{|x|}{a(\theta (\lambda )-t)}%
+C_{\lambda }
\end{equation*}%
for all $t\in \lbrack 0,\theta (\lambda ))$ and all $x\in \mathbb{R}$, where
$C_{\lambda }\geq 0$ and the map $\lambda \mapsto C_{\lambda }$ is
continuous.
\end{enumerate}

\begin{lemma}
\label{lm2} Under assumptions \emph{(h1)-(h3)}, for any sequence of positive
real numbers $\left( \lambda _{k}\right) $ converging to some $\lambda >0$,
and for every $\varepsilon >0$, there exists $\delta =\delta (\varepsilon
)>0 $ and $k_{0}=k\left( \varepsilon \right)$ such that
\begin{equation*}
\left\vert \int_{\theta \left( \lambda _{k}\right) -\delta }^{\theta \left(
\lambda _{k}\right) }u_{\lambda _{k}}(s)\,ds\right\vert <\varepsilon\quad
\text{for all }k\geq k\left( \varepsilon \right) .
\end{equation*}
\end{lemma}

\begin{proof}
Assume the contrary. Then, we may find a sequence $\left( \lambda
_{k}\right) $ with $\lambda _{k}>0$ for all $k$, $\lambda _{k}\rightarrow
\lambda >0$ as $k\rightarrow \infty $, and $\varepsilon >0$ such that for
any $\delta ,$ let it be $\delta _{n}=\frac{1}{n}$ and any $k_{0},$ say $%
k_{0}=n$ $\left( n\in \mathbb{N}\right) ,$ there is $k_{n}\geq n$ such that
\begin{equation}
\left\vert \int_{\theta \left( \lambda _{k_{n}}\right) -\delta _{n}}^{\theta
\left( \lambda _{k_{n}}\right) }u_{\lambda _{k_{n}}}(s)\,ds\right\vert \geq
\varepsilon .  \label{f1}
\end{equation}%
From (\ref{forma integrala}) and {(h3)}, one has,
\begin{equation}
\left\vert u_{\lambda _{k_{n}}}(t)\right\vert \leq \mathcal{C}+\int_{0}^{t}%
\frac{|u_{\lambda _{k_{n}}}(s)|}{a(\theta \left( \lambda _{k_{n}}\right) -s)}%
\,ds,\label{2.9}
\end{equation}%
for all $n$ and all $t\in \lbrack 0,\,\theta (\lambda _{k_{n}}))$, where
\begin{equation*}
\mathcal{C}=\sup_{n\in \mathbb{N}}\left\{ |u_{0}(\lambda _{k_{n}})|+\theta
\left( \lambda _{k_{n}}\right) C_{\lambda _{k_{n}}}\right\} <\infty .
\end{equation*}%
Applying Gr\"{o}nwall's lemma to (\ref{2.9}) yields
\begin{align*}
\left\vert u_{\lambda _{k_{n}}}(t)\right\vert & \leq \mathcal{C}\exp \left(
\int_{0}^{t}\frac{1}{a(\theta \left( \lambda _{k_{n}}\right) -s)}\,ds\right)
\\
& =\mathcal{C}\frac{\theta \left( \lambda _{k_{n}}\right) ^{1/a}}{(\theta
\left( \lambda _{k_{n}}\right) -t)^{1/a}},
\end{align*}%
for all $t\in \lbrack 0,\theta \left( \lambda _{k}\right) )$. Consequently,
\begin{equation}
\left\vert \int_{\theta \left( \lambda _{k_{n}}\right) -\delta _{n}}^{\theta
\left( \lambda _{k_{n}}\right) }u_{\lambda _{k_{n}}}(s)\,ds\right\vert \leq
\Tilde{\mathcal{C}}\delta _{n}^{\frac{a-1}{a}},  \label{f2}
\end{equation}%
where
\begin{equation*}
\Tilde{\mathcal{C}}=\mathcal{C}\frac{a}{a-1}\sup_{n\in \mathbb{N}}\,\left\{
\,\theta \left( \lambda _{k_{n}}\right) ^{1/a}\right\} .
\end{equation*}%
The conclusion follows immediately passing to limit in (\ref{f2}) with $%
n\rightarrow \infty $, which leads to a contradiction with (\ref{f1}). The required proof is thus completed.
\end{proof}

\begin{theorem}
\label{th_p} Assume that conditions \emph{(h1)-(h3)} are satisfied. Then,
the mapping $\varphi :\left( 0,\infty \right) \rightarrow
\mathbb{R}
,$
\begin{equation*}
\varphi :\left( 0,\infty \right) \rightarrow
\mathbb{R}, \, \, \varphi (\lambda )=\int_{0}^{\theta (\lambda )}u_{\lambda
}(s)\,ds
\end{equation*}%
is well-defined and continuous.
\end{theorem}

\begin{proof}
Note that for each $\lambda >0$, the integral
\begin{equation*}
\int_{0}^{\theta (\lambda )}u_{\lambda }(s)ds
\end{equation*}%
is convergent. Indeed, following the same steps as in the proof of Lemma \ref%
{lm2}, we conclude that for each $\lambda >0$, there are positive real
numbers $c_{1},c_{2}$ such that
\begin{equation}
\left\vert u_{\lambda }(t)\right\vert \leq \frac{c_{1}}{(\theta (\lambda
)-t)^{1/a}}+c_{2},  \label{eq88}
\end{equation}%
for all $t\in \left[ 0,\,\theta (\lambda )\right) .$ Consequently,
integrating relation \eqref{eq88} from $0$ to $\theta (\lambda )$ gives
\begin{equation*}
\int_{0}^{\theta (\lambda )}|u_{\lambda }(s)|ds\leq c_{3}\theta (\lambda )^{%
\frac{a-1}{a}}+c_{4}<\infty ,
\end{equation*}%
where $c_{3},c_{4}$ are constants dependent on $\lambda $. Thus $\varphi $
is well-defined.

To prove its continuity, consider any sequence $(\lambda _{k})$ with $%
\lambda _{k}>0$ for all $k$, $\lambda _{k}\rightarrow \lambda >0$ as $%
k\rightarrow \infty $, and any $\varepsilon >0$. Then, from Lemma \ref{lm2},
there exists $\delta ^{\prime }>0$ and $k_{0}$ such that
\begin{equation}
\left\vert \int_{\theta\left(  \lambda_{k}\right)  -\delta^{\prime}}%
^{\theta\left(  \lambda_{k}\right)  }u_{\lambda_{k}}\left(  s\right)
ds\right\vert <\varepsilon\ \quad\text{for all }k\geq k_{0}.\label{eq7}%
\end{equation}

\noindent Additionally, since the integral given $\varphi \left( \lambda \right) $ is
convergent, we can find $\delta ^{\prime \prime }>0$ such that \cite[pp. 51]%
{analizaII}
\begin{equation}
\left\vert \int_{\theta (\lambda )-\delta ^{\prime \prime }}^{\theta
(\lambda )}u_{\lambda }\left( s\right) ds\right\vert <\varepsilon .
\label{eq6}
\end{equation}%
Denote $\delta :=\min \{\delta ^{\prime },\delta ^{\prime \prime }\}$. Based
on Lemma \ref{lema_aux1}, the mapping
\begin{equation*}
\lambda \mapsto \int_{0}^{\theta (\lambda )-\delta }u_{\lambda }\left(
s\right) ds
\end{equation*}%
is continuous. Consequently, we find $k_{1}$ such that
\begin{equation}
\left\vert \int_{0}^{\theta \left( \lambda _{k}\right) -\delta }u_{\lambda
_{k}}\left( s\right) ds-\int_{0}^{\theta (\lambda )-\delta }u_{\lambda
}\left( s\right) ds\right\vert \leq \varepsilon\quad\text{ for all }k\geq k_{1}.
\label{eq5}
\end{equation}%
Thus, from (\ref{eq7}), (\ref{eq6}), (\ref{eq5}), one has
\begin{align*}
\left\vert \int_{0}^{\theta \left( \lambda _{k}\right) }u_{\lambda
_{k}}\left( s\right) ds-\int_{0}^{\theta (\lambda )}u_{\lambda }\left(
s\right) ds\right\vert & \leq \varepsilon +\left\vert \int_{\theta \left(
\lambda _{k}\right) -\delta }^{\theta \left( \lambda _{k}\right) }u_{\lambda
_{k}}\left( s\right) ds\right\vert +\left\vert \int_{\theta (\lambda
)-\delta }^{\theta (\lambda )}u_{\lambda }\left( s\right) ds\right\vert \\
& \leq 3\varepsilon ,
\end{align*}
for all $k\geq \max \left\{ k_{0},k_{1}\right\} .$ This proves that $\varphi
\left( \lambda _{k}\right) \rightarrow \varphi \left( \lambda \right) .$ The
conclusion is now immediate since the convergent sequence $\left( \lambda
_{k}\right) $ was arbitrarily chosen.
\end{proof}

Theorem \ref{th_p} leads immediately to the following controllability result
when a lower solution and an upper one are known.

\begin{theorem}\label{th 26}
Assume that conditions \emph{(h1)-(h3)} hold and that there exist $%
\underline{\lambda },\overline{\lambda }>0$ such that
\begin{equation*}
\varphi \left( \underline{\lambda }\right) <p\quad \text{and}\quad \varphi
\left( \overline{\lambda }\right) >p.
\end{equation*}%
Then there is $\lambda ^{\ast },$ intermediate between $\underline{\lambda }$
and $\overline{\lambda },$ such that $\psi _{\lambda ^{\ast }}(u_{\lambda
^{\ast }})=p.$
\end{theorem}

\begin{proof}
The conclusion follows directly by applying Darboux's intermediate value
theorem to the continuous function $\varphi $ (see, e.g., \cite[Theorem~4.23]%
{rudin}).
\end{proof}

\section{Approximate solving of the control problem}

Starting from the lower and upper solutions $\underline{\lambda },$ $%
\overline{\lambda },$ one can approximate $\lambda ^{\ast }$ by using the
following algorithm. For a similar use of this method, we refer the reader to \cite{algoritm}.

\begin{algorithm}[Bisection algorithm] \phantom{ } 
\newline
\noindent \textit{Step 0 (initialization)}: $k:=0,\ \underline{\lambda }_{0}:=\underline{%
\lambda },$ $\overline{\lambda }_{0}:=$ $\overline{\lambda }$.

\noindent \textit{Step} $k\ \left( k\geq 1\right) :$ compute $\lambda :=\frac{\underline{%
\lambda }_{k-1}+\overline{\lambda }_{k-1}}{2};$

\begin{itemize}

\item[$\bullet$]
If $\varphi \left( \lambda \right) =p,$ then $\lambda ^{\ast }=\lambda $ and
we are finished;

\item[$\bullet$]
If $\varphi \left( \lambda \right) <p,$ then set $\underline{\lambda }%
_{k}:=\lambda $ and $\overline{\lambda }_{k}:=\overline{\lambda }_{k-1}$ and
repeat Step $k$ with $k:=k+1;$

\item[$\bullet$]
If $\varphi \left( \lambda \right) >p,$ then set $\underline{\lambda }_{k}:=%
\underline{\lambda }_{k-1}$ and $\overline{\lambda }_{k}:=\lambda $ and
repeat Step $k$ with $k:=k+1;$

\end{itemize}

\noindent \textit{Stop criterion}: if $\left\vert \varphi \left( \lambda \right) -p\right\vert
\leq \varepsilon ,$ then $\lambda \simeq \lambda ^{\ast }$ (with error $%
\varepsilon $).

\end{algorithm}
We note that this step-by-step algorithm iteratively approximates the control solution. At each step, based on the obtained feedback, either the subsolution or the supersolution is improved.

\begin{theorem}
Under assumptions \emph{(h1)-(h3)}, the bisection algorithm is convergent to
a solution $\lambda ^{\ast }$ of the control problem.
\end{theorem}

\begin{proof}
If the algorithm does not stop after a finite number of steps, then it
generates two bounded and monotone (so convergent) sequences $\left(
\underline{\lambda }_{k}\right) $ and $\left( \overline{\lambda }_{k}\right)
,$ which in addition satisfy%
\begin{equation}
\left\vert \overline{\lambda }_{k}-\underline{\lambda }_{k}\right\vert =%
\frac{\left\vert \overline{\lambda }-\underline{\lambda }\right\vert }{2^{k}}%
\ \ \left( k\geq 0\right) ,  \label{a1}
\end{equation}%
\begin{equation}
\varphi \left( \underline{\lambda }_{k}\right) <p,\ \ \ \varphi \left(
\overline{\lambda }_{k}\right) >p.  \label{a2}
\end{equation}%
From (\ref{a1}), the two sequences have the same limit denoted $\lambda
^{\ast },$ while from (\ref{a2}), in virtue of the continuity Theorem \ref%
{th_p}, we obtain
\begin{equation*}
\varphi \left( \lambda ^{\ast }\right) \leq p\ \ \ \text{and\ \ \ }\varphi
\left( \lambda ^{\ast }\right) \geq p.
\end{equation*}%
Hence $\varphi \left( \lambda ^{\ast }\right) =p$ as desired.
\end{proof}

\begin{examB}
A typical example of function $f$ satisfying conditions (h1)-(h3) is
\begin{equation}\label{f}
f(t,x,\lambda )=\frac{x}{a(\lambda -t)},
\end{equation}%
where $a>1$. Clearly, $\theta (\lambda )=\lambda $. If in addition we take $%
u_{0}(\lambda )=\lambda ^{-\frac{1}{a}}$, we obtain the unique solution of
problem \eqref{pb principala},
\begin{equation*}
u_{\lambda }(t)=\frac{1}{(\lambda -t)^{\frac{1}{a}}}.
\end{equation*}%
Therefore,
\begin{equation*}
\varphi (\lambda )=\int_{0}^{\lambda }u_{\lambda }(s)\,ds=\frac{a}{a-1}%
\lambda ^{\frac{a-1}{a}},
\end{equation*}%
which is well-defined and continuous on $(0,\infty )$.

\begin{figure}[h]
    \centering
    \includegraphics[width=0.8\textwidth]{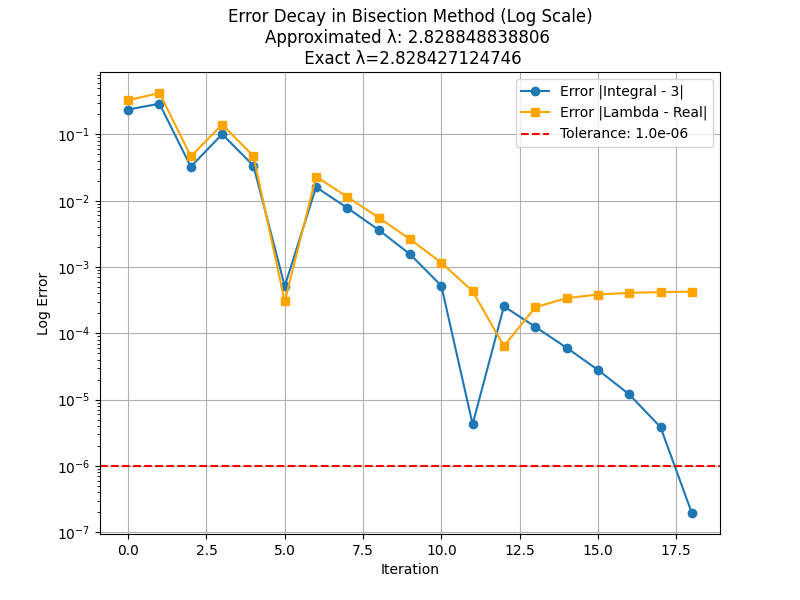}
    \caption{Error decay in the bisection method.}
    \label{fig:bisection_error}
\end{figure}

We conclude this example with some numerical simulations for the function \( f \) defined in \eqref{f}, with \( a = 3 \). Our aim is to determine \( \lambda^\ast \) such that \( \varphi(\lambda^\ast) = 3 \). The exact value is known to be \( \lambda_{\text{exact}} = 2^{\frac{3}{2}} \). For the lower and upper solutions of the control problem, we take \( \underline{\lambda} = 1 \) and \( \overline{\lambda} = 4 \), respectively, while the tolerance \( \varepsilon \) is chosen to be \( \varepsilon=10^{-6} \).

In the Figure 1, the blue curve represents the error between \( \varphi(\lambda) \) and $p=3$, where \( \lambda \) at each step $k$ is $\lambda := \frac{\underline{\lambda}_{k-1} + \overline{\lambda}_{k-1}}{2},$
while the orange curve represents the difference between the calculated value of  \( \lambda \) and the exact value \( \lambda_{\text{exact}}\). After 18 iterations, the approximate value of the control  is found to be
\[
\lambda^\ast = 2.828848838806.
\]

In Figure 2, the graph of \( u_\lambda \) is plotted for the last three  values of \( \lambda \) obtained from the bisection algorithm (those corresponding to the lowest error in the previous figure). 
We see that for $\lambda=2.82$, the graph of the function $u_\lambda$ almost overlaps with  the graph of $u_{\lambda_{\text{exact}}}$.
\begin{figure}[h]
    \centering
    \includegraphics[width=0.8\textwidth]{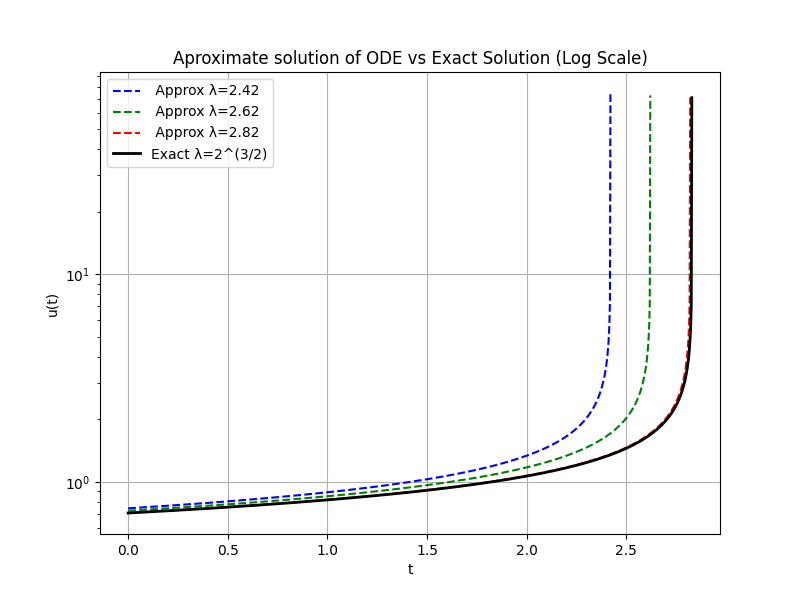}
    \caption{Approximate solution $u_\lambda$ vs exact solution $u_{\lambda_{\text{exact}}}$}
    \label{fig:bisection_error}
\end{figure}

\end{examB}

\begin{remark}
The conclusion of Theorem \ref{th_p} clearly remains valid under assumptions
(h1)-(h3), if instead of the functionals $\psi _{\lambda }$ we consider the
functionals
\begin{equation*}
\Bar{\psi}_{\lambda }(u)=\int_{0}^{\theta (\lambda )}|u(s)|\,ds,
\end{equation*}%
and instead of $\varphi $ we correspondingly take
\begin{equation*}
\Tilde{\varphi}(\lambda )=\int_{0}^{\theta (\lambda )}|u_{\lambda }(s)|\,ds.
\end{equation*}%
Moreover, we can extend this result to $L_{p}$ type functionals of the form
\begin{equation*}
\Tilde{\psi}_{\lambda }(u)=\left( \int_{0}^{\theta (\lambda
)}|u(s)|^{p}\,ds\right) ^{\frac{1}{p}},
\end{equation*}%
where $1\leq p<\infty $. In this case, if we replace (h3) by condition
\end{remark}

\begin{enumerate}
\item[\textbf{(h3')}] There exists a constant $a>p$ such that, for all $%
\lambda >0$, one has
\begin{equation*}
\left\vert f(t,x,\lambda )\right\vert \leq \frac{|x|}{a(\theta (\lambda )-t)}%
+C_{\lambda },
\end{equation*}%
for all $t\in \lbrack 0,\theta (\lambda ))$ and all $x\in \mathbb{R}$, where
$C_{\lambda }\geq 0$ and the map $\lambda \mapsto C_{\lambda }$ is
continuous.
\end{enumerate}

\section{Extension to fractional differential equations}
 The above results can be generalized to fractional differential equations with moving singularities. Such problems more accurately describe various physical, biological, or medical processes (see, e.g., \cite{samko,kilbas,hamani,podlubny}). Therefore, our results related to problem \eqref{pb principala} can be extended to the following problem:\begin{align*}
\begin{cases}
         ^c D^\alpha u(t) = f(t,u(t),\lambda), \quad t \in [0, \theta(\lambda)) \\
     u(0) = u_0(\lambda),
\end{cases}
 \end{align*}
 where $ ^c D^\alpha$ is the Caputo fractional derivative and $0<\alpha<1$.
As shown in the literature, the above problem is equivalent to the Voltera integral equation \begin{equation*}
    u_\lambda(t)=u_0(\lambda)+\frac{1}{\Gamma(\alpha)}\int_0^t (t-s)^{\alpha-1}f(s,u(s),\lambda)ds.
\end{equation*}
Note that, in our case, the control problem remain unchanged, i.e., find $\lambda$ such that \begin{equation*}
    \int_0^{\theta(\lambda)} u_\lambda(s)ds=p.
\end{equation*}
Since our entire analysis is grounded in the integral form of the Cauchy problem, we can easily extend the proof steps to address this more general case. The flexibility of the integral formulation allows for the adaptation of our methods without significant modifications. By imposing conditions similar to those outlined in (h1) and (h2), we can rigorously establish the controllability of the problem.

\section{Conclusions}

The analyzed control problem in this paper is atypical in several aspects: (a) it refers to equations with singularity; (b) the singularity itself depends on the control variable; (c) the controllability condition involves the moving singularity. All these aspects make the analysis much more complex and adapted to the specifics of the problem. The working techniques can also be taken into account for the investigation of other types of singular equations and controllability conditions including singular partial differential equations (see, e.g., \cite{movingsolution}). We believe and anticipate that the ideas and techniques used in this article will have the high degree of suitability for the specifics of each individual problem in future research.


\begin{thebibliography}{999}
\bibitem{tores}
Torres, P.J. {\it Mathematical Models with Singularities:
A Zoo of Singular Creatures};  Atlantis Press, Amsterdam, 2015. 
\bibitem{oregan} O'Regan, D. {\it Theory of Singular Boundary Value Problems}; World Scientiﬁc Publishing, River Edge, NJ,. USA, 1994.

\bibitem{doi} Callegari, A; Nachman, A. A nonlinear singular boundary value problem in the theory of pseudoplasitc fluids. \textit{SIAM J. Appl. Math.} \textbf{1980}, 38, 275--282.



\bibitem{moving1} Gingold, H. Rosenblat, S. Differential equations with
moving singularities. {\it SIAM J. Math. Anal} {\bf 1976}, 7(6), 942--957.

\bibitem{moving2} Gingold, H. Introduction to differential equations with
moving singularities. {\it Rocky Mountain J. Math.} {\bf 1976},
6(4), 571--574.

\bibitem{movingsolution} Fila, M; Takahashi, J; Yanagida, E. Solutions
with moving singularities for a one-dimensional nonlinear diffusion
equation. {\it Math. Ann.} {\bf 2024},
https://doi.org/10.1007/s00208-024-02882-0.




\bibitem{analizaII} Nicolescu, M; Dinculeanu, N; Marcus, S. {\it Manual de
analiz\u{a} matematic\u{a}a (II)}; Editura Didactic\u{a} \c{s}i Pedagogic\u{a}%
: Bucure\c{s}ti, Romania, 1964. (In Romanian)



\bibitem{pc} Haplea, I.\c{S}; Parajdi, L.G; Precup, R. On the
controllability of a system modeling cell dynamics related to
leukemia. {\it Symmetry} {\bf 2021}, 13, 1867.

\bibitem{pc2} Precup, R. On some applications of the controllability
principle for fixed point equations. {\it Results Appl. Math. } {\bf
2022}, 13, 100236.

\bibitem{coron} Coron, J.M. {\it Control and Nonlinearity}; AMS: Providence, 2007.


\bibitem{Coddington} Coddington, E.A. {\it An Introduction to Ordinary
Differential Equations}; Dover: New York, 1961.
\bibitem{pp} Precup, R. {\it Ordinary Differential Equations}; De Gruyter: Berlin,
2018.
\bibitem{p} O'Regan, D; Precup, R. {\it Theorems of Leray-Schauder Type and
Applications}; CRC Press, 2002.


\bibitem{ciarlet} Ciarlet, G. {\it Linear and Nonlinear Functional Analysis with
Applications}; SIAM: Philadelphia, 2013.




\bibitem{herve} Le Dret, H. {\it Nonlinear Elliptic Partial Differential
Equations}; Springer: Berlin, 2018.

\bibitem{gronwall} Gronwall, T.H. Note on the derivatives with respect to a
parameter of the solutions of a system of differential equations.
{\it Ann. Math.} {\bf 1919}, 20(4), 292--296.


\bibitem{ineq} Pachpatte, B.G.; Ames, W.F. {\it Inequalities for Differential
and Integral Equations}; Academic Press, 1997.






\bibitem{rudin} Rudin, W. {\it Principles of Mathematical Analysis}; MacGraw-Hill,
1976.


\bibitem{algoritm} Parajdi, L.G; Precup, R; Haplea, I.S. A method of lower and upper solutions for control problems and application to a model of bone marrow transplantation. {\it Int. J. Appl. Math. Comput. Sci}. \textbf{2023}, 33(3), 409 - 418. 

\bibitem{samko} Samko, S.G.; Kilbas, A.A.; Marichev, O.I. {\it Fractional Integrals and Derivatives: Theory and Applications}; Gordon and Breach Science Publishers: Yverdon, 1993.

\bibitem{kilbas} Kilbas, A.A.; Trujillo, J.J. Differentiale equations of fractional order: methods, results and problems-I. {\it Appl. Anal.} {\bf 2001}, 78, 153--192.
\bibitem{hamani} Benchohra, M.; Hamani, S.; Ntouyas, S.K. Boundary value problems for differential equations with fractional order and nonlocal conditions. {\it Nonlinear Anal.} {\bf 2009}, 71(7-8), 2391--2396.

\bibitem{podlubny} Podlubny, I. {\it Fractional Differential Equations}; Mathematics in Sciences and Engineering, Vol. 198; Academic Press: San Diego, 1999.


\end{thebibliography}
\end{document}